\documentclass[12pt]{article}
 \usepackage{a4wide}

\usepackage{fancyhdr, amsmath, amssymb, amsfonts, url, mathrsfs, graphicx, epsfig,  amsthm}
\usepackage{tikz}
\usepackage{amscd}

\numberwithin{equation}{section}

\AtEndDocument{\bigskip{\footnotesize
  \noindent
  \textsc{D. Parmenter, School of Mathematics, University of Bristol, Bristol, BS8 1UG, UK, and the Heilbronn Institute for Mathematical Research, Bristol, UK.}
  }}

\newtheorem{thm}{Theorem}[section]
\newtheorem{cor}[thm]{Corollary}
\newtheorem{definition}[thm]{Definition}

\newtheorem{remark}[thm]{Remark}
\title{Local large deviation principle for Smale spaces}
\author{David Parmenter}
\date{September 2025}

\begin{document}

\maketitle

\abstract{Large deviation principles for hyperbolic systems are well studied and provide exponential rates for the deviations of Birkhoff averages from their limit. This short article presents a local large deviation principle for Smale spaces, in particular studying the rate functions of deviations with respect to conditional Gibbs measures supported on local unstable manifolds. The proof builds on a result due to Kifer and pressure growth estimates due to Parmenter and Pollicott.}

\section{Introduction and main results}
%\textcolor{red}{Add introductory discussion of large deviation principles in dynamics.}

%We will study `local' large deviation principles for Smale spaces. 
Consider the compact metric space $X$ and homeomorphism $f : X \rightarrow X$. We call the system a Smale space if there is a local product structure given by a bracket map $[\cdot,\cdot]$ which is analogous to the Bowen bracket for Axiom A diffeomorphisms. In this sense Smale spaces generalise the locally maximal uniformly hyperbolic diffeomorphisms. Smale spaces also include subshifts of finite type and therefore provide a unified setting to study both classes of systems.

Recall a result due to Ruelle \cite{ruelle} that the SRB measure $\mu_{SRB}$ for uniformly hyperbolic attractors determines the Birkhoff average of almost every point with respect to the volume $m$. 
%can be characterised by the convergence of the Birkhoff sums for almost every point with respect to the volume. 
In particular,  
%there is a set $V \subset U$ (where $U$ is the basin of attraction in Definition \ref{attractordef}) having full Lebesgue measure such that, 
for every continuous observable $\phi : X \rightarrow \mathbb{R}$ and $m$-a.e. $x$,
    \begin{equation*}
        \frac{1}{n} S_n \phi(x) = \frac{1}{n} \sum_{i=0}^{n-1} \phi(f^i x) \rightarrow \int_X \phi d\mu_{SRB}.
    \end{equation*}

 %In the light of Theorem \ref{youngSRB}, 
%Theorem \ref{youngSRB} establishes convergence for Lebesgue almost every point and Theorem \ref{O-P} is a way to quantify the convergence.

Our motivation is in the study of the rate of exponential decay of the measure of the set below,
\begin{equation*} \label{set}
   \lambda \bigg( \bigg\{y : \frac{1}{n}S_n\phi(y) \in L\bigg\} \bigg) \text{\hspace{5mm} as $n \rightarrow \infty$},
\end{equation*}
where $L \subset \mathbb{R}$ is a Borel set and $\lambda$ is an appropriately chosen measure. 
%In our setting $\lambda$ will be a measure supported on a piece of unstable manifold with a conditional Gibbs property for some H\"older potential. See Definition \ref{condgibbs}.

%There have been many dynamical large deviation principles and we with focus on an application of 
%but we shall only mention the most relevant for our results. 
%Using Markov partitions Orey and Pelikan \cite{orey1989deviations} proved a large deviation principle for Anosov diffeomorphisms with $\lambda$ being the normalised volume on $X$. 
%The following is the first large deviation theorem applied to the setting of Anosov diffeomorphisms due to Orey and Pelikan \cite{orey1989deviations}.
%An important dynamical large deviation principle is due to Kifer \cite{kifer1990large} which we make substantial use of. 
Before stating the main result we define an orbital measure.
\begin{definition} 
For $y \in X$ and $n \in \mathbb{N}$ the $n$-th orbital measure $\zeta_y^n$ is defined by,
\begin{equation*}
\zeta_y^n = \frac{1}{n} \sum_{i=0}^{n-1} \delta_{f^iy}.
\end{equation*}
\end{definition}
Let $\mathcal{M}(X)$ be the set of probability measures on $X$ and $\mathcal{M}_f(X)$ be the $f$-invariant probability measures on $X$. For some H\"older continuous $G : X \rightarrow \mathbb{R}$, denote by $\mu_{x,G}^u$ the conditional Gibbs measure (Definition \ref{condgibbs}) supported on a piece of local unstable manifold $W_\delta^u(x)$ (Definition \ref{def:unstable}) centred at $x$. Heuristically a conditional Gibbs measure has the Gibbs property for unstable Bowen balls. Then we have the following local large deviation principle.

\begin{thm} \label{LLDP}
Let $(X, d, f, [\cdot , \cdot])$ be a mixing Smale space and $G : X \rightarrow \mathbb{R}$ a H\"older continuous potential. For $\mu_{G}$ a.e. $x \in X$, $\delta > 0$ sufficiently small and any closed $K \subset \mathcal{M}(X)$,
\begin{equation*}
\limsup_{n \rightarrow \infty} \frac{1}{n} \log \mu_{x,G}^u(\{y : \zeta_y^n \in K\}) \leq - \inf \{I(\nu) : \nu \in K\},
\end{equation*}
where,
\begin{equation} \label{rate func}
I(\nu) = 
    \begin{cases}
      P(G) - \int G d\nu - h_\nu(f) 			& \text{$\nu \in \mathcal{M}_f(X)$}\\
      \infty								 	& \text{$\nu \notin \mathcal{M}_f(X)$}.
    \end{cases} 
\end{equation}
Additionally, for any open $J \subset \mathcal{M}(X)$,
\begin{equation*}
\liminf_{n \rightarrow \infty} \frac{1}{n} \log \mu_{x,G}^u(\{y : \zeta_y^n \in J\}) \geq - \inf\{I(\nu) : \nu \in J\}.
\end{equation*}
\end{thm}

It is a classical result that the above holds when $\mu_{x,G}^u$ is replaced by the Gibbs measure $\mu_G$, see for example Kifer \cite{kifer1990large}. Theorem \ref{LLDP} is local in the sense that not only does the large deviation principle hold globally but we show that it also holds with the same rate function when we restrict to almost every unstable leaf.

%Lastly, we finish with a corollary of Theorem \ref{LLDP} analogous to Corollary 2.1 in \cite{kifer1990large}. 

In addition we also deduce the following corollary. Let $\phi \in C(X)$ and define $\tilde{\phi} : \mathcal{M}(X) \rightarrow \mathbb{R}$ by $\tilde{\phi}(\mu) = \int \phi d \mu$ so that $\tilde{\phi}(\zeta_y^n) = S_n \phi(y)$. It is easily seen from Theorem \ref{LLDP} that the following type two large deviation principle holds.

\begin{cor}
Let $(X, d, f, [\cdot , \cdot])$ be a mixing Smale space, $\phi : X \rightarrow \mathbb{R}$ be continuous and $G:X \rightarrow \mathbb{R}$ be H\"older. For $\mu_{G}$ a.e. $x \in X$, $\delta > 0$ sufficiently small and any closed set $L \subset \mathbb{R}$,
\begin{equation*}
\limsup_{n \rightarrow \infty} \frac{1}{n} \log \mu_{x,G}^u\bigg(\bigg\{y : \frac{1}{n}S_n\phi(y) \in L\bigg\}\bigg) \leq - \inf \{\tilde{I}(L) : \alpha \in L\},
\end{equation*}
and for any open set $U \subset \mathbb{R}$,
\begin{equation*}
\liminf_{n \rightarrow \infty} \frac{1}{n} \log \mu_{x,G}^u\bigg(\bigg\{y : \frac{1}{n}S_n\phi(y) \in U\bigg\}\bigg) \geq - \inf \{\tilde{I}(U) : \alpha \in U\},
\end{equation*}
where
\begin{equation*}
    \tilde{I}(\alpha) =  \inf_{\nu}   
    \begin{cases}
      P(G) - \int G d\nu - h_\nu(f) 			& \text{if $\nu \in \mathcal{M}_f(X)$ and $\int \phi d\nu = \alpha$,}\\
      \infty								 	& \text{\text{otherwise}}.
    \end{cases} 
\end{equation*}
\end{cor}

%This can be stated in slightly higher generality as in Corollary 2.1 \cite{kifer1990large} but for ease of reading this is the form we choose.

In fact we remark in \textsection 3 that Theorem \ref{LLDP} also holds for the class of partially hyperbolic diffeomorphisms considered in \cite{parmenter2022constructing} and for the continuous time setting of hyperbolic flows.

%taking $\lambda$ to be the equilibrium state conditioned on $W_\delta^u(x)$. We show that for almost every $x \in X$ restricting to a piece of unstable manifold through $x$ there is no change to the rate function found by Kifer \cite{kifer1990large}.}
%for the exponential decay of the set in equation (\ref{set}).

 %\section{Results}

 \section{Smale spaces}
  
%We work in the setting of Smale spaces. 
Heuristically Smale spaces are compact metric spaces which satisfy a local product structure by stable/unstable manifolds determined by some homeomorphism. Let $X$ be a non-empty compact metric space with metric $d$. Assume there is an $\epsilon > 0$ and a map, $[\cdot , \cdot]$ with the following properties:
\begin{equation*}
[\cdot , \cdot] : \{ (x,y) \in X \times X \text{ : } d(x,y)<\epsilon\} \rightarrow X
\end{equation*}
is a continuous map such that $[x,x] = x$ and 
\begin{align*}
[[x,y] , z] & = [x,z], \tag{SS1} \\
[x,[y,z]] & =[x,z], 	\tag{SS2}		\\
f([x,y]) & = [f(x),f(y)],	\tag{SS3}
\end{align*}
when the two sides of these relations are defined. 

Additionally, we require the existence of a constant $0 < \lambda < 1$ such that for any $x \in X$ we have the following two conditions:
For $y,z \in X$ such that $d(x,y),d(x,z) < \epsilon$ and $[y,x] = x = [z,x]$, we have
\begin{equation*}
d(f(y),f(z)) \leq \lambda d(y,z); \tag{SS4}
\end{equation*}
and for $y,z \in X$ such that $d(x,y),d(x,z) < \epsilon$ and $[x,y] = x = [x,z]$, we have
\begin{equation*}
d(f^{-1}(y),f^{-1}(z)) \leq \lambda d(y,z). \tag{SS5}
\end{equation*}

\begin{definition} \label{smaledef}
Let $X$ be a compact metric space with metric $d$. Let $f : X \rightarrow X$ be a homeomorphism and $[\cdot , \cdot]$ have the properties $SS1-SS5$ above. Then we define the  Smale space to be the quadruple $(X, d, f, [\cdot , \cdot])$.
If $f : X \rightarrow X$ is also topological mixing then we call $(X, d, f, [\cdot , \cdot])$ a mixing Smale space.
\end{definition}

The product map given by $[\cdot,\cdot]$ can be used to define local stable and unstable manifolds.

\begin{definition} \label{def:unstable}
For   sufficiently small $\delta > 0$ one can define the stable and unstable manifolds through $x \in X$ by 
    \begin{align*}
        W_\delta^s(x) & = \{y \in X \text{ : } y = [x,y] \text{ and } d(x,y)<\delta\}, \\
        W_\delta^u(x) & = \{y \in X \text{ : } y = [y,x] \text{ and } d(x,y)<\delta\}.
    \end{align*}
\end{definition}

It is easily seen from $SS4$ and $SS5$ that Definition \ref{def:unstable} can be equivalently characterised in terms of the behaviour of forward and backward orbits.

Important examples of Smale spaces include locally maximal uniformly hyperbolic diffeomorphisms and subshifts of finite type. For explanation of these examples in the context of Smale spaces, see \cite{parmenter2024constructing}. For further details about Smale spaces, see \cite{ruelle2004thermodynamic} \textsection 7.

\section{Conditional Gibbs measures}

The conditional Gibbs property is an analogue of the Gibbs property bounding the measure of dynamic Bowen balls restricted to unstable leaves.
\begin{definition} \label{condgibbs}
 For $y \in \textcolor{black}{W_\delta^u(x)}$, $0 < \epsilon < \delta$ and $n \in \mathbb{N}$
we define the  {unstable Bowen ball} of radius $\epsilon$ by
 $$B_{d_u}(y,n,\epsilon) 
 = \{z \in W^u(x) \hbox{ : } d_u(f^iz, f^iy) < \epsilon \hbox{ for } 0 \leq i \leq n-1\}
 $$ 
 where $d_u$ is the induced unstable metric on $W_\delta^u(x)$. 

Let $\mu^u$ be a measure supported on a piece of unstable manifold \textcolor{black}{centred at $x$}. We say that it has the conditional Gibbs property for $G:X \rightarrow \mathbb{R}$ if for every small $\epsilon > 0$ there is a constant $K = K(\epsilon) > 0$ such that, for every $y \in W_\delta^u(x)$ and $n \in \mathbb{N}$ we have,
\begin{equation*}
K^{-1} \leq \frac{\mu^u(B_{d_u}(y, n, \epsilon))}{e^{S_nG(y) - nP(G)}} \leq K.
\end{equation*}
We write $\mu^u = \mu_{G}^u$ if this conditional property holds. We may also write $\mu_{x,G}^u$ when we need to emphasise the measure is supported on a piece of unstable manifold centred at $x$.
\end{definition}

%Bowen \cite{} shows that for expansive systems with the specification property 

For $G:X \rightarrow \mathbb{R}$ H\"older it can be shown using, for example, \cite{bowen-periodic} and \cite{ruelle2004thermodynamic} that the equilibrium state $\mu_G$ has the local product structure. Intuitively this means that `locally' (restricting to the rectangle $R = [W_\delta^u(x), W_\delta^s(x)]$) $\mu_G$ can be decomposed into measures on unstable and stable leaves, $\mu_G |_R \approx \mu_x^u \otimes \mu_x^s$ and $\mu_x^u$ has the conditional Gibbs property for $G$. For a more detailed discussion see \cite{parmenter2024constructing}.

%\textcolor{red}{Need to add discussion of what these measures are. In particular, mention that the Gibbs measure has a local product structure and the conditional measures on unstable leaves satisfies the conditional Gibbs property.}

\section{Proof of Theorem \ref{LLDP}}

%We will apply Theorem 2.1 \cite{kifer1990large} and Lemma 4.1 \cite{parmenter2024constructing} to prove a local large deviation principle. 

%Replacing the volume (or Gibbs measure) in Theorem 3.5 \cite{kifer1990large} by a conditional Gibbs measure we show a large deviation principle holds without changing the rate function from \cite{kifer1990large}, even when we restrict our attention to a small piece of unstable manifold. 

Kifer \cite{kifer1990large} shows that uniqueness of equilibrium state for an appropriate set of continuous potentials along with the existence of a characterisation of the topological pressure in terms of a growth estimate, 
\begin{equation*}
    P(\phi) = \lim_{n \rightarrow \infty} \frac{1}{n} \log \int e^{S_n \phi(y)} d \lambda(y),
\end{equation*}
is enough to establish large deviation principles. %In particular \cite{kifer1990large} considers the setting of locally maximal $C^2$ uniformly hyperbolic diffeomorphisms with $\lambda$ either the normalised volume or an equilibrium state for a H\"older potential. 
The proof of Theorem \ref{LLDP} relies on Theorem 2.1 \cite{kifer1990large} along with the growth estimate in Lemma 4.1 \cite{parmenter2024constructing}.

%We consider another application of \cite{kifer1990large} to prove a large deviation principle for Smale spaces. Using Lemma 4.1 \cite{parmenter2024constructing} we prove a local large deviation principle where $\lambda$ is a measure supported on a piece of unstable manifold with the conditional Gibbs property (Definition \ref{condgibbs}) for some H\"older potential.
%Kifer \cite{kifer1990large} considers $\lambda$ either being the volume on $X$ or an equilibrium state. 
%Our result is local in the sense that for almost every $x \in X$ even if we restrict the support of $\lambda$ to a piece of unstable manifold through $x$, denoted $W_\delta^u(x)$, 
%by considering an equilibrium state conditioned on $W_\delta^u(x)$ 
%we still recover a large deviation with the same rate function as Theorem 3.1 \cite{kifer1990large}.

\begin{proof}[Proof of Theorem \ref{LLDP}]
Lemma 4.1 \cite{parmenter2024constructing} proves that for any continuous $\varphi : X \rightarrow \mathbb{R}$,
\begin{equation*}
P(\varphi) = \lim_{n \rightarrow \infty} \frac{1}{n} \log \int_{W_\delta^u(x)} e^{S_n(\varphi - G)(y) + nP(G)} d\mu_{x,G}^u(y).
\end{equation*}
Defining $Q : C(X) \rightarrow \mathbb{R}$ by,
\begin{align*}
Q(\varphi) & = \lim_{n \rightarrow \infty} \frac{1}{n} \log \int_{W_\delta^u(x)} e^{S_n\varphi(y)} d\mu_{x,G}^u(y).
\end{align*}
Then by Lemma 4.1 \cite{parmenter2024constructing},
\begin{align*}
	Q(\varphi) = P(\varphi + G) - P(G).
\end{align*}
Using an extension of the variational principle (\cite{kifer1990large}, Theorem 3.1),
\begin{equation} %\label{Q}
Q(\varphi) = \sup_{\mu \in \mathcal{M}_f(X)} \bigg(\int (\varphi + G) d\mu + h_\mu(f) - P(G) \bigg),
\end{equation}
or $Q(\varphi) = - \infty$ if $\mathcal{M}_f(X) = \emptyset$.

The entropy function is affine and upper semi-continuous, see \cite{misiurewicz1976topological}. %by Bowen \cite{bowen1972entropy} and Newhouse \cite{newhouse}. 
Therefore, by the duality theorem (\cite{aubin2006applied}, pp 201) the rate function in (\ref{rate func}) is convex conjugate to $Q(\varphi)$. Therefore, using Theorem 2.1 \cite{kifer1990large},
\begin{equation*}
\limsup_{n \rightarrow \infty} \frac{1}{n} \log \mu_{x,G}^u(\{y : \zeta_y^n \in K\}) \leq - \inf  \{ I(\nu) \text{ $: $ $\nu \in K$}\}.
\end{equation*}
The lower bound follows similarly using the second half of Theorem 2.1 \cite{kifer1990large}.
\end{proof}

%Moreover, let $\tilde{I}(\alpha) = \inf \{I(\nu) \text{ : } \tilde{G} (\nu) = \alpha\}$ if $\alpha$ is in the image of $\tilde{f}$ and $\tilde{I}(\alpha) = \infty$ otherwise.

We finish with the following remarks about other systems that can be considered using the methods in the proof of Theorem \ref{LLDP}.

\begin{remark}
    Taking $G=\varphi^{geo}$ we have $\mu_{x,G}^u = Vol_{W_\delta^u(x)}$. By Proposition 3.3 %and 6.4 
    \cite{parmenter2022constructing} we can deduce analogous local large deviation principles for partially hyperbolic diffeomorphisms satisfying growth conditions in the centre-stable manifolds. 
    %\textcolor{red}{need to be careful as we need to have uniqueness of equilibrium state}
    %Need to add remark about partially hyperbolic systems and the fact we can use the centre isometry growth lemma.
\end{remark}

\begin{remark}
    We end with a remark that Theorem 2.1 \cite{kifer1990large} also considers the continuous time setting. 
    %Kifer \cite{kifer1990large} proves Theorem \ref{Kifer} in the continuous time setting as well. 
    Moreover, it is clear we can use Lemma 6.5 \cite{PaPo} along with Theorem 2.1 \cite{kifer1990large} to prove a local large deviation principle analogous to Theorem \ref{LLDP} for hyperbolic flows.  
\end{remark}

\bibliographystyle{abbrv}

%\bibliography{bibliography}

\begin{thebibliography}{111}

\bibitem{aubin2006applied}
J.-P. Aubin and I. Ekeland. {\it Applied nonlinear analysis}. Courier Corporation, 2006.

\bibitem{bowen-periodic}
R. Bowen. Periodic points and measures for {A}xiom {A} diffeomorphisms. Trans. Amer. Math. Soc, 154:377-397, 1971.

\bibitem{kifer1990large}
Y. Kifer. Large deviations in dynamical systems and stochastic processes. Trans. Amer. Math. Soc, 321(2):505-524, 1990. 

\bibitem{misiurewicz1976topological}
M. Misiurewicz. Topological conditional entropy. {Studia Mathematica}, 55(2):175-200, 1976.

\bibitem{PaPo}
D. Parmenter and M. Pollicott. Gibbs measures for  hyperbolic attractors defined  by  densities, Discrete and Continuous Dynamical Systems 42, (2022) 3953-3977.

\bibitem{parmenter2024constructing}
D. Parmenter and M. Pollicott. Constructing equilibrium states for Smale spaces.
arXiv preprint arXiv:2403.04646, 2024.

\bibitem{parmenter2022constructing}
 D. Parmenter and M. Pollicott. Constructing equilibrium states for some partially
hyperbolic attractors via densities. arXiv preprint arXiv:2210.10701, 2024.

\bibitem{ruelle}
D. Ruelle. A measure associated with Axiom A attractors. American Journal of
Mathematics, pages 619–654, 1976.

\bibitem{ruelle2004thermodynamic}
D. Ruelle. {\it Thermodynamic formalism: the mathematical structure of equilibrium statistical mechanics}, Cambridge University Press, 2004.


    
\end{thebibliography}

\end{document}